\documentclass[12pt,reqno]{article}

\usepackage{amsmath,amssymb,amsthm}

\usepackage{pictex}

\DeclareMathOperator{\Ext}{Ext}

\DeclareMathOperator{\dimv}{\textbf{dim}}
\newcommand{\MATRIX}[4]{\left[\begin{array}{cc}
#1 & #2 \\#3 & #4 \end{array}\right]} 

\newtheorem{Thm}{Theorem}[section]

\newtheorem{Cor}[Thm]{Corollary}
\newtheorem{Prop}[Thm]{Proposition}

\begin{document}
\title{Categorification of the Fibonacci Numbers\\ Using Representations of Quivers}
\author{Philipp Fahr\\
Brehmstr. 51\\
D-40239 D\"usseldorf, Germany\\
\texttt{philfahr@gmx.net}\\
\\Claus Michael Ringel \\ 
Fakult\"at f\"ur Mathematik, Universit\"at Bielefeld\\ POBox 100 131\\ D-33501 Bielefeld, Germany\\
\texttt{ringel@math.uni-bielefeld.de}}\date{}
\maketitle

\begin{abstract} In a previous paper we have presented a partition formula for the 
even-index Fibonacci numbers using the preprojective representations of the 
$3$-Kronecker quiver and its universal
cover, the $3$-regular star. Now we deal in a similar way with the 
odd-index Fibonacci numbers.
The Fibonacci modules introduced here provide a convenient categorification of 
the Fibonacci numbers.
\end{abstract}

\section{Introduction}

We consider the Fibonacci numbers with $f_0 = 0,$ $f_1 = 1$
and the recursion rule $f_{i+1} = f_i + f_{i-1}$ for $i \ge 1.$ As suggested by many authors,
one may use this recursion rule also for $i \le 0$ and one obtains in this way
Fibonacci numbers $f_i$ for all integral indices $i$ (Knuth \cite{Kn}  
calls this extended set the negaFibonacci numbers); the sequence $f_t$ with $-10\le t \le 10$
looks as follows:
$$
   -55,34,-21,13,-8,5,-3,2,-1,1,0,1,1,2,3,5,8,13,21,34,55; 
$$
in general, $f_{-t} = (-1)^{t+1} f_t.$ This rule, but also many other observations, show
that Fibonacci numbers have quite different properties depending whether the index is even
or odd, and the present report will strongly support this evidence. 

Our aim is to outline a categorification
of the Fibonacci numbers, or better of Fibonacci pairs, using the
representation theory of quivers. By definition, {\it Fibonacci pairs} are the pairs
of the form $[f_t,f_{t+2}]$ or $[f_t,f_{t-2}]$; such a pair will be 
called {\it even} or {\it odd} provided the
index $t$ of $f_t$ is even or odd, respectively. 

It is well-known that the dimension vectors of some indecomposable representations
of the $3$-Kronecker quiver $Q$ are Fibonacci pairs. In the previous note  \cite{FR},
preprojective and preinjective  representations were
used in order to derive a partition formula for the even-index Fibonacci numbers $f_{2n}$. 
In
the present paper we will exhibit a corresponding partition formula for the odd-index
Fibonacci numbers $f_{2n+1}$. Now 
regular representations of $Q$ play a role (see also \cite{C} and \cite{R}). Actually,   
we will be dealing directly with the universal cover of $Q$, thus with graded representations
of $Q$. The universal cover of $Q$ is the quiver $(T,E,\Omega)$, where $(T,E)$ is 
the 3-regular tree and $\Omega$ is a bipartite orientation. In
section 4 we will introduce some representations $P_t(x)$ and $R_t(x,y)$
of the quiver $(T,E,\Omega)$
which we label Fibonacci modules. The Fibonacci modules 
will be used in section 4 and 6
in order to categorify the Fibonacci pairs (and hence the Fibonacci numbers).
The dimension vectors of the Fibonacci modules are called
Fibonacci vectors, they are  studied already in secton 2 and 3 and provide 
the partition formula.

Categorification  of a set of numbers means to consider instead of
these numbers suitable objects in a category (here representations of quivers),
so that the numbers in question 
occur as invariants of the objects (see for example \cite{CY}). 
Equality of numbers
may be visualized  by isomorphisms of objects, functional relations by
functorial ties. We will see that certain addition formulas for Fibonacci numbers
can be interpreted very well by displaying  filtrations of Fibonocci modules. 

The partition formula for the odd-index Fibonacci numbers presented here
is due to Fahr, see \cite{F}, the remaining considerations are
based on discussions of the authors during the time when
Fahr was a PhD student at Bielefeld, the final version was
written by Ringel.

\section{Fibonacci Vectors}

Let us start with the $3$-regular tree $(T,E)$ with vertex set $T$ and edge set $E$.
We consider paths in this tree, thus sequences 
$x_0,x_1,\dots, x_t$ of elements $x_i$ of $T$ such that all the sets $\{x_{i-1},x_i\}$
belong to $E$, where $1\le i \le t$, and $x_{i+1} \neq x_{i-1}$ for $0 < i < t$;
if such a sequence exists, then we say that
$x_0$ and $x_t$ have distance $t$ and write $d(x_0,x_t) = t$. As usual, vertices with distance
$1$ are said to be neighbors. 

Let $K_0(T)$ be the free abelian group generated by the vertex set $T$, or better by elements $s(x)$
corresponding bijectively to the vertices $x\in T.$ 
The elements of $K_0(T)$ can be written as finite sums $a = \sum_{x\in T} a_xs(x)$ with $a_x \in \Bbb Z$
(and almost all $a_x = 0$), or just as $a = (a_x)_{x\in T}$; the set of vertices $x$ with $a_x \neq 0$ will
be called the support of $a$. 

We are going to single out suitable elements of $K_0(T)$
which we will call Fibonacci vectors: 
For every pair $(x,y)$ with $\{x,y\} \in E$, let
$r(x,y) = s(x)+s(y).$ 
The Fibonacci vectors will be  obtained from the elements $s(x)$ and $r(x,y)$ 
by applying sequences of reflections. 
For $y\in T$ the reflection
$\sigma^y: K_0(T) \to K_0(T)$ is defined 
as follows: For $a\in K_0(T)$, its image $\sigma^ya$ is given by
$(\sigma^ya)_x = a_x$ for $x\neq y,$ and $(\sigma^ya)_y = -a_y + \sum_{\{x,y\}\in E} a_x.$
Observe that the reflections $\sigma^y, \sigma^{y'}$ commute provided $y,y'$ are not neighbors.

We denote by $\Sigma^x$ the composition of all the reflections $\sigma^y$ with $d(x,y)$ 
being even (note that this composition is defined: first of all, since any element $a\in K_0(T)$
has finite support, the number of vertices $y$ with $\sigma^ya \neq a$ is finite; second, given two
different vertices $y,y'$ such that the distance between $x$ and $y$, as well as between $x$ and $y'$
is even, then $y,y'$ cannot be neighbors, thus $\sigma^y, \sigma^{y'}$ commute).
Similarly, we denote  by $\underline\Sigma^x$ the composition of all the reflections 
$\sigma^y$ with $d(x,y)$ being odd (again, this is well-defined).

We define $s_t(x)$ for all $t\ge 0,$ using induction:
First of all, $s_0(x) = s(x)$.
If $s_t(x)$ is already defined for some $t\ge 0$, let $s_{t+1}(x) = \underline\Sigma^xs_t(x)$ in case
$t$ is even, and  $s_{t+1}(x) = \Sigma^xs_t(x)$ in case
$t$ is odd, thus
 \begin{eqnarray*}
 s_0(x) &=& s(x), \\
 s_1(x) &=& \underline\Sigma^xs(x),\\
 s_2(x) &=& \Sigma^x\underline\Sigma^xs(x),\\
 s_3(x) &=& \underline\Sigma^x\Sigma^x\underline\Sigma^xs(x),\\
  &\dots&
\end{eqnarray*}

Similarly, we define $r_t(x,y)$ for all $t\ge 0$: 
First of all, $r_0(x,y) = r(x,y)$.
If $r_t(x,y)$ is already defined for some $t\ge 0$, let $r_{t+1}(x,y) = \underline\Sigma^xr_t(x,y)$ 
in case $t$ is even, and  $r_{t+1}(x,y) = \Sigma^xr_t(x,y)$ in case
$t$ is odd. The elements $s_t(x)$ with $t \ge 0$ as well as the elements  $r_t(x,y)$ 
with   $t\in\mathbb Z$ will be called the {\em Fibonacci vectors.}

Define
$$
 s_t(x)_- = \sum_{d(x,z)\not\equiv t} s_t(x)_z,\quad \text{and}\quad
 s_t(x)_+ = \sum_{d(x,z)\equiv t} s_t(x)_z,\quad 
$$
where $\equiv$ means equivalence modulo $2$, 
and similarly,
$$
 r_t(x,y)_- = \sum_{d(x,z)\not\equiv t} r_t(x,y)_z,\quad \text{and}\quad
 r_t(x,y)_+ = \sum_{d(x,z)\equiv t} r_t(x,y)_z. 
$$

\begin{Prop} 
{The Fibonacci vectors have non-negative coordinates and
$$
  s_t(x)_- = f_{2t}, \
  s_t(x)_+ = f_{2t+2},\ r_t(x,y)_- = f_{2t-1},\ r_t(x,y)_+ = f_{2t+1}.
$$
}
\end{Prop}

\begin{proof} The proof for $s_t$ has been given in the previous paper \cite{FR}.
In the same way, one
deals with $r_t(x,y).$ But note that in \cite{FR}
we have fixed some base point $x_0$, and we wrote $\Phi_0, \Phi_1$
instead of  $\Sigma^{x_0}$, $\underline\Sigma^{x_0}$ respectively. Also, there, we
have denoted the
function $s_{2t}(x_0)$ by $a_t$, and the functions $s_{2t+1}(x_0)$ by $\Phi_1a_t$.
\end{proof}

The functions $s_t(x_0)$ have been exhibited in \cite{FR}, for
$0 \le t \le 5.$ At the end of the present paper we present the
functions $r_t(x,y)$ for $0\le t \le 5.$ 
	        
\section{The Partition Formula}

In the previous paper \cite{FR} the functions $s_t(x)$ were used in order to provide
a partition formula for the even-index Fibonacci numbers (see also \cite{P, H}).
We now consider the functions
$r_t(x,y)$ in order to obtain a corresponding partition formula for the odd-index Fibonacci numbers. In order to get all odd-index Fibonacci numbers, 
it is sufficient to look at the functions $r_{2t}(x,y)$, since
$$
 f_{4t-1}   =  r_{2t}(x,y)_- , \quad
 f_{4t+1}   =  r_{2t}(x,y)_+ .
$$

For any $t$, the numbers $s_t(x)_z$ only
depend on the distance $d(x,z)$. In contrast, for $r_t(x,y)$ and $d(x,z) > 0,$ there are now
{\bf two} values which occur as $r_t(x,y)$ 
depending on whether the path connecting $x$ and $z$
runs through $y$ or not. Let us denote by $T_s(x,y)$ the set of vertices $z$ of $T$ with
$d(x,z) = s$ and such that the path connecting $x, z$ does not involve $y$, and 
$T'_s(x,y)$ shall denote the set of vertices $z$ of $T$ with
$d(x,z) = s,$ such that the path connecting $x, z$ does involve $y$.
Clearly, we have for $s \ge 1$
\begin{equation}
 |T_s(x,y)| = 2^{s}, \qquad |T'_s(x,y)| = 2^{s-1}.  \tag{1}
\end{equation}

The proof of $(1)$ is by induction: Of course,  $T'_1(x,y) = \{y\}$
whereas $T_1(x,y)$ consists of the two remaining neighbors of $x$. For any vertex $z$
with $d(x,z) = s > 0$, there are precisely two neighbors $z', z''$ such that $d(x,z') = d(x,z'') 
= s+1,$ therefore  $|T_{s+1}(x,y)| = 2|T_s(x,y)|$ and $|T_{s+1}'(x,y)| = 2|T_s'(x,y)|$.
  \medskip

Thus, let us look at the function $u_t:\Bbb Z \to \Bbb N$ defined for $s\ge 0$ by
$u_t[s] = r_{2t}(x,y)_z$ with $z\in T_s(x,y)$, and for $s \le -1$ by
$u_t[s] = r_{2t}(x,y)_z$ with $z\in T'_{-s}(x,y)$.
Using the equality (1), we see:
\begin{eqnarray*}
 f_{4t-1}   &=&  r_{2t}(x,y)_-  = \sum_{d(x,z) \text{ odd}} r_{2t}(x,y)_z \\
   &=& \sum_{s \ge 0 \text{ odd}} |T_s(x,y)|u_t[s] + 
       \sum_{s > 0 \text{ odd}} |T'_s(x,y)|u_t[-s] \\
   &=& \sum_{s \ge 0 \text{ odd}} 2^su_t[s] + 
       \sum_{s > 0 \text{ odd}} 2^{s-1}u_t[-s] \\
\\
 f_{4t+1}   &=&  r_{2t}(x,y)_+ = \sum_{d(x,z) \text{ even}} r_{2t}(x,y)_z \\
   &=& \sum_{s \ge 0 \text{ even}} |T_s(x,y)|u_t[s] + 
       \sum_{s > 0 \text{ even}} |T'_s(x,y)|u_t[-s] \\
   &=& \sum_{s \ge 0 \text{ even}} 2^su_t[s] + 
       \sum_{s > 0 \text{ even}} 2^{s-1}u_t[-s] 
\end{eqnarray*}

Without reference to $r_{2t}$, we can define the functions $u_t$ directly as follows:
We start with $u_0[0] = u_0[-1] = 1$ and $u_0[s] = 0$ for all other $s$.
If $u_t$ is already defined for some $t\ge 0$, then we define first $u_{t+1}[s]$ for
odd integers $s$ by the rule
$$
  u_{t+1}[s] = \left\{\begin{matrix} 2u_t[s-1]-u_t[s]+u_t[s+1] & \text{for} & s < 0 \cr
                       u_t[s-1]-u_t[s] + 2u_t[s+1] & \text{for} & s > 0 \end{matrix} \right.
$$
and in a second step for even integers $s$ by
$$
  u_{t+1}[s] = \left\{\begin{matrix} 2u_{t+1}[s-1]-u_t[s]+u_{t+1}[s+1] & \text{for} & s < 0 \cr
                       u_{t+1}[s-1]-u_t[s]+2u_{t+1}[s+1] & \text{for} & s \ge 0 \end{matrix} \right.
$$

Here is the table of the numbers $u_t[s]$, for  $t\le 4$; in addition, we list the corresponding
Fibonacci numbers $f_{-1}, f_1, f_3, \dots, f_{17}.$
$$
{\beginpicture
\setcoordinatesystem units <.7cm,.5cm>
\plot -3 9.7 16.5 9.7 /

\setdashes <1mm>

\plot -2.2 4.5  -2.2 10.8 /
\plot 14.0 4.5  14.0 10.8 /
 
\setsolid
\plot -3.1 10.8  -2.2 9.7 /
\put{$t$} at -3.2 10.2
\put{$s$} at -2.5 10.8

\put{$-6$} at -1 10.3
\put{$-5$} at 0 10.3
\put{$-4$} at 1 10.3
\put{$-3$} at 2 10.3
\put{$-2$} at 3 10.3
\put{$-1$} at 4 10.3
\put{$0$} at 5 10.3
\put{$1$} at 6 10.3
\put{$2$} at 7 10.3
\put{$3$} at 8 10.3
\put{$4$} at 9 10.3
\put{$5$} at 10 10.3
\put{$6$} at 11 10.3
\put{$7$} at 12 10.3
\put{$8$} at 13 10.3
\put{$\dots$} [r] at 14 10.3
\put{$\dots$} [r] at -1.3 10.3

\put{$f_{4t-1}$} at 14.7 10.3
\put{$f_{4t+1}$} at 16 10.3

\put{$1$} [r] at 15.2 9
\put{$1$} [r] at 16.7 9
\put{$2$} [r] at 15.2 8
\put{$5$} [r] at 16.7 8
\put{$13$} [r] at 15.2  7
\put{$34$} [r] at 16.7 7
\put{$89$} [r] at 15.2 6
\put{$233$} [r] at 16.7 6
\put{$610$} [r] at 15.2 5
\put{$1597$} [r] at 16.7 5

\put{$0$} [r] at -2.7 9
\put{$1$} [r] at 4.1 9
\put{$1$} [r] at 5.1 9

\put{$1$} [r] at -2.7 8
\put{$1$} [r] at 5.1 8
\put{$1$} [r] at 6.1 8
\put{$1$} [r] at 7.1 8

\put{$2$} [r] at -2.7 7
\put{$1$} [r] at 3.1 7
\put{$1$} [r] at 4.1 7
\put{$4$} [r] at 5.1 7
\put{$2$} [r] at 6.1 7
\put{$3$} [r] at 7.1 7
\put{$1$} [r] at 8.1 7
\put{$1$} [r] at 9.1 7

\put{$3$} [r] at -2.7 6
\put{$1$} [r] at 1.1 6
\put{$1$} [r] at 2.1 6
\put{$6$} [r] at 3.1 6
\put{$5$} [r] at 4.1 6
\put{$17$} [r] at 5.1 6
\put{$8$} [r] at 6.1 6
\put{$13$} [r] at 7.1 6
\put{$4$} [r] at 8.1 6
\put{$5$} [r] at 9.1 6
\put{$1$} [r] at 10.1 6
\put{$1$} [r] at 11.1 6
\put{$4$} [r] at -2.7 5
\put{$1$} [r] at -.9 5
\put{$1$} [r] at 0.1 5
\put{$8$} [r] at 1.1 5
\put{$7$} [r] at 2.1 5
\put{$32$} [r] at 3.1 5
\put{$24$} [r] at 4.1 5
\put{$77$} [r] at 5.1 5
\put{$35$} [r] at 6.1 5
\put{$60$} [r] at 7.1 5
\put{$19$} [r] at 8.1 5
\put{$26$} [r] at 9.1 5
\put{$6$} [r] at 10.1 5
\put{$7$} [r] at 11.1 5
\put{$1$} [r] at 12.1 5
\put{$1$} [r] at 13.1 5

 \endpicture}
$$

{\bf Remark.} 
We can reformulate the recursion rules using the generalized Cartan matrix $A = (a_{ij})_{ij}$
(indexed over the integers) with Dynkin diagram
$$
{\beginpicture
\setcoordinatesystem units <1.5cm,.7cm>
\multiput{$\circ$} at -3 0  -2 0  -1 0  0 0  1 0  2 0  3 0 /
\put{$-3$} at -3 0.5 
\put{$-2$} at -2 0.5 
\put{$-1$} at -1 0.5 
\put{$0$} at 0 0.5 
\put{$1$} at 1 0.5 
\put{$2$} at 2 0.5 
\put{$3$} at 3 0.5 
\plot -.8 0  -.2 0 /
\plot -2.8 .1  -2.2 .1 /
\plot -1.8 .1  -1.2 .1 /
\plot .8 .1  .2 .1 /
\plot 1.8 .1  1.2 .1 /
\plot 2.8 .1  2.2 .1 /
 
\plot -2.8 -.1  -2.2 -.1 /
\plot -1.8 -.1  -1.2 -.1 /
\plot .8 -.1  .2 -.1 /
\plot 1.8 -.1  1.2 -.1 /
\plot 2.8 -.1  2.2 -.1 /
\plot -2.6 .2 -2.4 0  -2.6 -.2 / 
\plot -1.6 .2 -1.4 0  -1.6 -.2 /
 \plot .6 .2  .4 0    .6 -.2 / 
\plot 1.6 .2  1.4 0  1.6 -.2 / 
\plot 2.6 .2  2.4 0  2.6 -.2 /

\multiput{$\dots$} at -4 0  4 0 /
\endpicture}
$$
One pair of vertices is simply laced, for the remaining laced pairs, 
one of the numbers $a_{ij}, a_{ji}$ is
equal to $-1$, the other to $-2$ (as in the case of the Cartan matrix of type $\Bbb B_2$). 
In order to obtain $u_{t+1}$ from $u_t$, apply first the reflections $\sigma_i$ with $i$ odd,
then the reflections $\sigma_i$ with $i$ even.

\section{Fibonacci Modules}

We now endow $(T,E)$ with a bipartite orientation $\Omega$; in this way, we deal
with a quiver $(T,E,\Omega)$ and we consider its representations. 
Note that there are just two bipartite orientations for $(T,E)$.
Given a vertex $x$ of $T$, let us denote by $\Omega^x_t$ 
the bipartite orientation such that $x$ is a sink in case $t$ is even, and a source, in case
$t$ is odd. 
Let $P_t(x)$ be the indecomposable representation of $(T,E,\Omega^x_t)$ with dimension
vector $s_t(x)$. If there is an edge between $x$ and $y$, let $R_t(x,y)$ be the 
indecomposable representation of $(T,E,\Omega^x_t)$ with dimension
vector $r_t(x,y)$; it is well-known that these modules exist and are unique up to
isomorphism: for $t=0$ this assertion is trivial, and the Bernstein-Gelfand-Ponomarev
reflection functors \cite{BGP} can be used in order to construct $P_{t+1}(x)$ from
$P_t(x)$, as well as $R_{t+1}(x,y)$ from $R_t(x,y).$

We call the modules $P_t(x)$ and $R_t(x,y)$ {\em Fibonacci modules}.

\begin{Prop} Let $x$ be a vertex of $T$ with neighbors $y,y',y''$ and $t \ge 1$
an integer. Assume that for even $t$, the vertex $x$ is a sink, and  if $t$ is odd, $x$ is a source. 
Then there are exact sequences
$$
0 \to P_{t-1}(y) \to P_t(x) \to R_t(x,y) \to 0,
$$
$$
0 \to P_{t-1}(y') \to R_t(x,y) \to R_{t-1}(y'',x) \to 0,
$$
and they are unique up to ismomorphism. 
\end{Prop}

\begin{proof} This is clear for $t=1$ and follows by induction using the
Bernstein-Gelfand-Ponomarev reflection functors.
\end{proof}

\begin{Cor}
Let $x_0,x_1,\dots,x_t$ be a path with $x_0$ a sink. Then $P_t(x_t)$ has a
filtration 
$$
 P_0(x_0) \subset P_1(x_1) \subset \cdots \subset P_t(x_t)
$$
with factors
$$
 P_i(x_i)/P_{i-1}(x_{i-1}) = R_i(x_i)
$$
for $1 \le i \le t.$
\end{Cor}

\begin{proof}
 Since $x_0$ is a sink, we see that $x_i$ is a sink for $i$ even and a source for
$i$ odd. Thus, the proposition provides exact sequences
$$
0 \to P_{i-1}(x_{i-1}) \to P_i(x_i) \to R_t(x_i,x_{i-1}) \to 0,
$$
for $1\le i \le t.$ 
\end{proof}

\begin{Cor}
Let $x_{-1},x_0,x_1,\dots,x_t,x_{t+1}$ be a path with $x_0$ a source. For $0 \le i \le t$, let
$z_i$ be the neighbor of $x_i$ different from $x_{i-1}$ and $x_{i+1}$. Then
there is an exact sequence
$$
 0 \to P_0(z_0)\oplus \cdots\oplus P_t(z_t) \to R_{t+1}(x_t,x_{t+1}) \to R_0(x_{-1},x_0) \to 0.
$$
\end{Cor}

\begin{proof}  
Here, $x_0$ is a source, thus we see that $x_i$ is a source for $i$ even 
and a sink for $i$ odd. The proof is by induction on $t$, the case $t=0$ should be clear.
Assume that $t \ge 1.$ We consider the vertices $x_0,\dots, x_{t+1}, z_1,\dots z_t$,
but deal with the opposite orientation $\Omega'$ 
(so that now $x_1$ is a source); the corresponding
respresentations will be distinguished by a dash, thus $P'_i(z_i)$ is the 
indecomposable representation of $(T,E,\Omega')$ with dimension
vector $s_i(z_i)$, and so on. By induction, there is an exact sequence
$$
 0 \to P'_0(z_1)\oplus \cdots \oplus P'_{t-1}(z_t) \to R'_{t}(x_t,x_{t+1}) \to R'_0(x_0,x_1) \to 0.
$$
Applying reflection functors (at all sinks), we obtain the exact sequence
$$
 0 \to P_1(z_1)\oplus \cdots \oplus P_t(z_t) \to R_{t+1}(x_t,x_{t+1}) \to R_1(x_0,x_1) \to 0.
$$
Note that $P_0(z_0)$ is a subrepresentation of $R_1(x_1,x_0),$ thus we obtain the
induced exact sequence
$$
 0 \to P_1(z_1)\oplus \cdots \oplus P_t(z_t) \to W \to P_0(z_0) \to 0,
$$
where $W$ is a subrepresentation of $R_{t+1}(x_t,x_{t+1})$ with $R_{t+1}(x_t,x_{t+1})/W$
being isomorphic to $R_1(x_1,x_0)/P_0(z_0) = R_0(x_{-1},x_0)$.

On the other hand, the induced sequence has to split, since $P_0(z_0)$ is projective,
thus $W$ is the direct sum of the representations $P_i(z_i)$ with $0\le i \le t.$
This completes the proof.
\end{proof}

Of course, proposition 4.1 should be seen as a categorification of the defining equation 
$$
 f_{t+1} = f_{t-1}+f_t
$$
for the Fibonacci numbers. Corollaries 4.2 and 4.3  are categorifications of the equalities
$$
 f_{2t} =  \sum_{i=1}^t f_{2i-1}, \qquad f_{2t+1} = 1 +\; \sum_{i=1}^t f_{2i}.
$$

{\bf Remark.} The representations $R_0(x_{-1},x_0),  P_0(z_0), \dots , P_t(z_t) $
are pairwise orthogonal bricks, the corresponding $\Ext$-quiver is
$$
{\beginpicture
\setcoordinatesystem units <1.5cm,.6cm>
\put{$R_0(x_{-1},x_0)$} at 1.5 2
\put{$P_0(z_0)$} at 0 0
\put{$P_1(z_1)$} at 1 0
\put{$\cdots$} at 2 0
\put{$P_t(z_t)$} at 3 0
\arrow <1.5mm> [0.25,0.75] from 1.3 1.6  to .1 0.4
\arrow <1.5mm> [0.25,0.75] from 1.5 1.6  to 1 0.4
\arrow <1.5mm> [0.25,0.75] from 1.7 1.6  to 2.9 0.4

\endpicture}
$$
Let $\mathcal F$ denote  the full subcategory  of of all representations of $(T,E,\Omega)$
with a filtration
with factors of the form  $R_0(x_{-1},x_0),  P_0(z_0), \dots , P_t(z_t) $, then corollary 4.3 shows that
$R_{t+1}(x_t,x_{t+1})$ is indecomposable projective in $\mathcal F$.

\section{Fibonacci Pairs}

The Fibonacci pairs are related to the integral quadratic form $q(x,y) = x^2+y^2-3xy$
as follows:

\begin{Prop}
{A pair $[x,y]\in \Bbb Z^2$ is a Fibonacci pair if and only if $|q(x,y)| = 1.$}
\end{Prop}

The pairs $[x,y]$ with $|q(x,y)| = 1$ form two hyperbolas. The hyperbola of all
pairs $[x,y]$ with $q(x,y) = 1$ yields the even Fibonacci pairs, these are the black dots on the
following picture. The hyperbola  $q(x,y) = -1$ yields the odd Fibonacci pairs, they are indicated by small circles.
$$
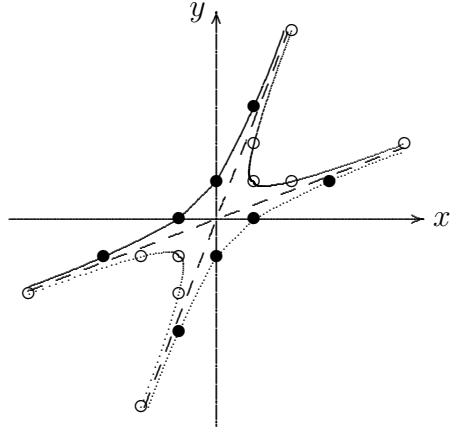
{\beginpicture
\setcoordinatesystem units <.5cm,.5cm>
\multiput{} at -5 -5  5 5 /
\multiput{$\bullet$} at 0 1  1 3   -1 0  -3 -1 0 -1  -1 -3   1 0  3 1      /  
\multiput{$\circ$} at -1 -1  -1 -2  -2 -1  -2 -5  -5 -2   1 1  1 2  2 1  2 5  5 2  /  
\arrow <1.5mm> [0.25,0.75] from -5.5 0 to 5.5 0
\arrow <1.5mm> [0.25,0.75] from  0 -5.5 to 0 5.5
\setdashes <2mm>
\plot -5 -1.91  5 1.91 / 
\plot -1.91 -5  1.91 5 /
\setsolid
\setquadratic
\plot 1.8 5  1 3  0 1  -.4 .6  -1 0  -3 -1  -5 -1.8 /
\plot 2 5  1 2  1 1  2 1  5  2 /

\setdots <.5mm>
\plot -1.8 -5  -1 -3  0 -1  .4 -.6  1 0  3 1  5 1.8 /
\plot -2 -5  -1 -2  -1 -1  -2 -1  -5  -2 /
\put{$x$} at  6 0
\put{$y$} at -.5 5.5
\endpicture}
$$

The quadratic form $q$ on $\mathbb Z^2$ defines the Kac-Moody root system
\cite{K1} 
for the Cartan matrix $\MATRIX{2}{-3}{-3}{2}$ (see \cite{FR}, section 6).
The pairs
$[x,y]$ with $q(x,y) = 1$ are the real roots, those with $q(x,y) \le 0$ the
imaginary roots. We identify $\mathbb Z^2$ with the
Grothendieck group of the $3$-Kronecker modules, and $q$ is the 
Euler form. Given a $3$-Kronecker module $M$, the corresponding element in the Grothendieck
group is its dimension vector $\dimv M$. 

\section{Recursion formula for Fibonacci numbers with index of the same parity.}

This is the recursion formula:
$$
 f_{t+2} = 3f_t - f_{t-2}. \qquad (1)
$$

\begin{proof} The term 
$f_{t+2} - 3f_t + f_{t-2}$ is just the sum of $f_{t+2}-f_{t+1}-f_t$, of 
$f_{t+1}-f_t-f_{t-1}$ and of $-(f_t-f_{t-1}-f_{t-2})$, thus equal to zero.
\end{proof}
 
Note that the recursion formula can also be rewritten as
$$
 f_{t-2} = 3f_t - f_{t+2}, \qquad (2)
$$
thus the same recurrence formula works both for going up and for going down. 
Also, we can write it in the form
$$
 f_{t-2}+f_{t+2} = 3f_t. \qquad (3)
$$

The formulae (1) and (2) are categorified by looking at the reflection functors $\sigma_+$
and $\sigma_-$: Given a $3$-Kronecker module $M = (M_1,M_2,\alpha,\beta,\gamma)$,
we obtain $(\sigma_+M)_2 = M_1$ and $(\sigma_+M)_1$  as the kernel of the map
$(\alpha,\beta,\gamma):M_1^3 \to M_2$; note that if $M$ is indecomposable and not the simple
projective module $P_0$, then this kernel has dimension $3\dim M_1 - \dim M_2.$ Similarly,
we obtain $(\sigma_-M)_1 = M_2$ and $(\sigma_+M)_2$  as the cokernel of the map
$(\alpha,\beta,\gamma)^t:M_1 \to M_2^3$; this cokernel has dimension $3\dim M_2 - \dim M_1$
provided $M$ is indecomposable and not simple injective. 
The formula (3) is categorified by the Auslander-Reiten sequences (see for example \cite{ARS})
$$
 0 \to P_{n-1} \to P_n^3 \to P_{n+1} \to 0, 
$$
$$
 0 \to Q_{n+1} \to Q_n^3 \to Q_{n-1} \to 0,
$$
as well as by exact sequences of the form 
$$
0 \to R_{n-1,\lambda} \to E(n,\lambda) \to R_{n+1,\lambda} \to 0
$$
with an indecomposable module $E(n,\lambda)$ having dimension vector $3\dimv R(n,\lambda).$


\hrule
\medskip

\noindent 2010 {\it Mathematics Subject Classification}:
Primary 11B39, Secondary 16G20.

\noindent \emph{Keywords: } 
Fibonacci numbers, Quiver representations, Universal cover, 3-regular tree, 3-Kronecker quiver.

\noindent (Concerned with sequences 
A000045,  A132262, A147316 of \cite{Sloane})

\medskip
\hrule
\medskip


{\bf Illustrations.}
The following illustrations show the 
functions $r_t(x,y)$ for $t=0,..,5$, with $x$ the center and $y$ above $x$.
The dotted circles indicate a fixed
distance to $x$. We show  $(T,E)$  with bipartite orientation such that $x$ is a sink (thus $z$ is a sink if $d(x,y)$ is even, and a source if $d(x,z)$ is odd).
The region containing the  arrow $y\to x$ has been
dotted. For the convenience of the reader we have indicated the corresponding Fibonacci pair
$[f_{2t-1},f_{2t+1}].$

$$
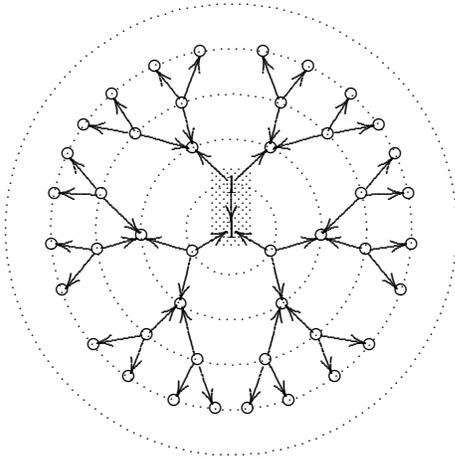
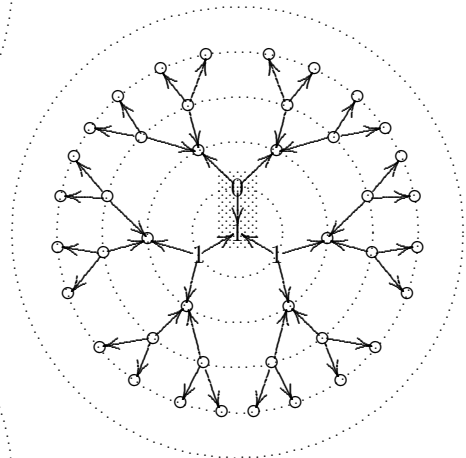
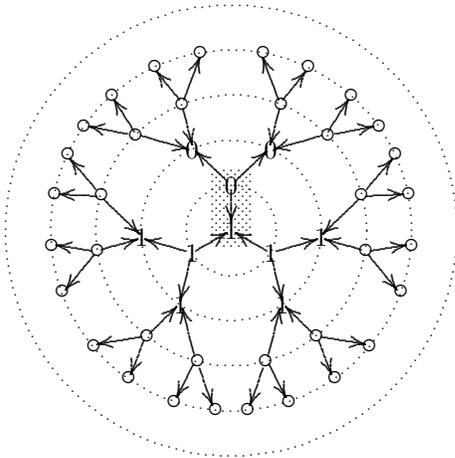
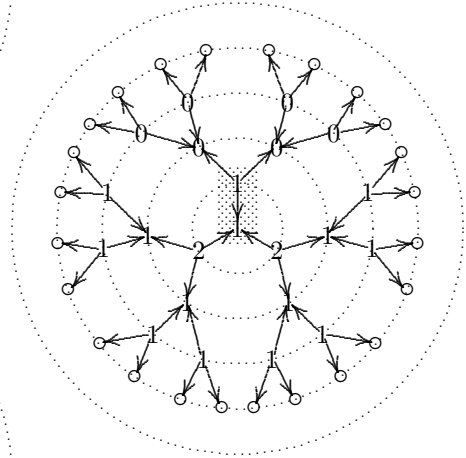
{\beginpicture
\setcoordinatesystem units <0.6cm,0.6cm>

\multiput{\beginpicture

\put{\footnotesize  1} at 0 0
\multiput{$\circ$} at -.866 -.5
  .866 -.5  /
\multiput{\footnotesize 1} at  0 1 /

\arrow <2mm> [0.25,0.75] from 0 0.9 to  0 0.2
\arrow <2mm> [0.25,0.75] from -.78 -.45
 to   -0.1 -0.05 
\arrow <2mm> [0.25,0.75] from .78 -.45
 to   0.1 -0.05 
\setshadegrid span <.5mm>
\vshade -.5 -.4 1.4  .5 -.4 1.4 /

\multiput{$\circ$} at -.866 1.8 .866 1.8 /
\arrow <2mm> [0.25,0.75] from -0.1 1.1  to  -.78 1.75
\arrow <2mm> [0.25,0.75] from  0.1 1.1  to   .78 1.75
\multiput{$\circ$} at -2.13 2.1  -1.1 2.8 
  2.13 2.1  1.1 2.8 /
\arrow <2mm> [0.25,0.75] from -2. 2.1     to  -.95 1.8 
\arrow <2mm> [0.25,0.75] from -1.07 2.7   to  -.866 1.9
\arrow <2mm> [0.25,0.75] from 2. 2.1      to  .95 1.8 
\arrow <2mm> [0.25,0.75] from 1.07 2.7    to  .866 1.9

\multiput{$\circ$} at -3.27 2.3  -2.65 3 -1.7 3.62  -.7 3.93  /
\multiput{$\circ$} at  3.27 2.3  2.65 3  1.7 3.62  .7 3.93 /

\arrow <2mm> [0.25,0.75] from -2.3 2.14   to  -3.15 2.3 
\arrow <2mm> [0.25,0.75] from -2.15 2.2    to  -2.6 2.9
\arrow <2mm> [0.25,0.75] from -1.15 2.9    to  -1.65 3.52
\arrow <2mm> [0.25,0.75] from -1.05 2.9    to  -.7 3.83 

\arrow <2mm> [0.25,0.75] from 2.3 2.14   to  3.15 2.3 
\arrow <2mm> [0.25,0.75] from 2.15 2.2    to  2.6 2.9
\arrow <2mm> [0.25,0.75] from 1.15 2.9    to  1.65 3.52
\arrow <2mm> [0.25,0.75] from 1.05 2.9    to  .7 3.83 

\startrotation by -0.5 0.866 about 0 0
\multiput{$\circ$} at -.866 1.8 .866 1.8 /
\arrow <2mm> [0.25,0.75] from -0.1 1.1  to  -.78 1.75
\arrow <2mm> [0.25,0.75] from  0.1 1.1  to   .78 1.75
\multiput{$\circ$} at -2.13 2.1  -1.1 2.8 
  2.13 2.1  1.1 2.8 /
\arrow <2mm> [0.25,0.75] from -2. 2.1     to  -.95 1.8 
\arrow <2mm> [0.25,0.75] from -1.07 2.7   to  -.866 1.9
\arrow <2mm> [0.25,0.75] from 2. 2.1      to  .95 1.8 
\arrow <2mm> [0.25,0.75] from 1.07 2.7    to  .866 1.9

\multiput{$\circ$} at -3.27 2.3  -2.65 3 -1.7 3.62  -.7 3.93  /
\multiput{$\circ$} at  3.27 2.3  2.65 3  1.7 3.62  .7 3.93 /

\arrow <2mm> [0.25,0.75] from -2.3 2.14   to  -3.15 2.3 
\arrow <2mm> [0.25,0.75] from -2.15 2.2    to  -2.6 2.9
\arrow <2mm> [0.25,0.75] from -1.15 2.9    to  -1.65 3.52
\arrow <2mm> [0.25,0.75] from -1.05 2.9    to  -.7 3.83 

\arrow <2mm> [0.25,0.75] from 2.3 2.14   to  3.15 2.3 
\arrow <2mm> [0.25,0.75] from 2.15 2.2    to  2.6 2.9
\arrow <2mm> [0.25,0.75] from 1.15 2.9    to  1.65 3.52
\arrow <2mm> [0.25,0.75] from 1.05 2.9    to  .7 3.83 
\stoprotation

\startrotation by -0.5 -0.866 about 0 0
\multiput{$\circ$} at -.866 1.8 .866 1.8 /
\arrow <2mm> [0.25,0.75] from -0.1 1.1  to  -.78 1.75
\arrow <2mm> [0.25,0.75] from  0.1 1.1  to   .78 1.75
\multiput{$\circ$} at -2.13 2.1  -1.1 2.8 
  2.13 2.1  1.1 2.8 /
\arrow <2mm> [0.25,0.75] from -2. 2.1     to  -.95 1.8 
\arrow <2mm> [0.25,0.75] from -1.07 2.7   to  -.866 1.9
\arrow <2mm> [0.25,0.75] from 2. 2.1      to  .95 1.8 
\arrow <2mm> [0.25,0.75] from 1.07 2.7    to  .866 1.9

\multiput{$\circ$} at -3.27 2.3  -2.65 3 -1.7 3.62  -.7 3.93  /
\multiput{$\circ$} at  3.27 2.3  2.65 3  1.7 3.62  .7 3.93 /

\arrow <2mm> [0.25,0.75] from -2.3 2.14   to  -3.15 2.3 
\arrow <2mm> [0.25,0.75] from -2.15 2.2    to  -2.6 2.9
\arrow <2mm> [0.25,0.75] from -1.15 2.9    to  -1.65 3.52
\arrow <2mm> [0.25,0.75] from -1.05 2.9    to  -.7 3.83 

\arrow <2mm> [0.25,0.75] from 2.3 2.14   to  3.15 2.3 
\arrow <2mm> [0.25,0.75] from 2.15 2.2    to  2.6 2.9
\arrow <2mm> [0.25,0.75] from 1.15 2.9    to  1.65 3.52
\arrow <2mm> [0.25,0.75] from 1.05 2.9    to  .7 3.83 
\stoprotation

\setdots <1mm>
\circulararc 360 degrees from 1 0 center at 0 0
\circulararc 360 degrees from 2 0 center at 0 0
\circulararc 360 degrees from 3 0 center at 0 0
\circulararc 360 degrees from 4 0 center at 0 0
\circulararc 360 degrees from 5 0 center at 0 0
\endpicture} at 0 22  /

\multiput{\beginpicture

\put{\footnotesize 1} at 0 0
\multiput{\footnotesize 1} at -.866 -.5
  .866 -.5  /
\multiput{\footnotesize 0} at  0 1 /

\arrow <2mm> [0.25,0.75] from 0 0.9 to  0 0.2
\arrow <2mm> [0.25,0.75] from -.78 -.45
 to   -0.1 -0.05 
\arrow <2mm> [0.25,0.75] from .78 -.45
 to   0.1 -0.05 

\setshadegrid span <.5mm>
\vshade -.5 -.4 1.4  .5 -.4 1.4 /

\multiput{$\circ$} at -.866 1.8 .866 1.8 /
\arrow <2mm> [0.25,0.75] from -0.1 1.1  to  -.78 1.75
\arrow <2mm> [0.25,0.75] from  0.1 1.1  to   .78 1.75
\multiput{$\circ$} at -2.13 2.1  -1.1 2.8 
  2.13 2.1  1.1 2.8 /
\arrow <2mm> [0.25,0.75] from -2. 2.1     to  -.95 1.8 
\arrow <2mm> [0.25,0.75] from -1.07 2.7   to  -.866 1.9
\arrow <2mm> [0.25,0.75] from 2. 2.1      to  .95 1.8 
\arrow <2mm> [0.25,0.75] from 1.07 2.7    to  .866 1.9

\multiput{$\circ$} at -3.27 2.3  -2.65 3 -1.7 3.62  -.7 3.93  /
\multiput{$\circ$} at  3.27 2.3  2.65 3  1.7 3.62  .7 3.93 /

\arrow <2mm> [0.25,0.75] from -2.3 2.14   to  -3.15 2.3 
\arrow <2mm> [0.25,0.75] from -2.15 2.2    to  -2.6 2.9
\arrow <2mm> [0.25,0.75] from -1.15 2.9    to  -1.65 3.52
\arrow <2mm> [0.25,0.75] from -1.05 2.9    to  -.7 3.83 

\arrow <2mm> [0.25,0.75] from 2.3 2.14   to  3.15 2.3 
\arrow <2mm> [0.25,0.75] from 2.15 2.2    to  2.6 2.9
\arrow <2mm> [0.25,0.75] from 1.15 2.9    to  1.65 3.52
\arrow <2mm> [0.25,0.75] from 1.05 2.9    to  .7 3.83 

\startrotation by -0.5 0.866 about 0 0
\multiput{$\circ$} at -.866 1.8 .866 1.8 /
\arrow <2mm> [0.25,0.75] from -0.1 1.1  to  -.78 1.75
\arrow <2mm> [0.25,0.75] from  0.1 1.1  to   .78 1.75
\multiput{$\circ$} at -2.13 2.1  -1.1 2.8 
  2.13 2.1  1.1 2.8 /
\arrow <2mm> [0.25,0.75] from -2. 2.1     to  -.95 1.8 
\arrow <2mm> [0.25,0.75] from -1.07 2.7   to  -.866 1.9
\arrow <2mm> [0.25,0.75] from 2. 2.1      to  .95 1.8 
\arrow <2mm> [0.25,0.75] from 1.07 2.7    to  .866 1.9

\multiput{$\circ$} at -3.27 2.3  -2.65 3 -1.7 3.62  -.7 3.93  /
\multiput{$\circ$} at  3.27 2.3  2.65 3  1.7 3.62  .7 3.93 /

\arrow <2mm> [0.25,0.75] from -2.3 2.14   to  -3.15 2.3 
\arrow <2mm> [0.25,0.75] from -2.15 2.2    to  -2.6 2.9
\arrow <2mm> [0.25,0.75] from -1.15 2.9    to  -1.65 3.52
\arrow <2mm> [0.25,0.75] from -1.05 2.9    to  -.7 3.83 

\arrow <2mm> [0.25,0.75] from 2.3 2.14   to  3.15 2.3 
\arrow <2mm> [0.25,0.75] from 2.15 2.2    to  2.6 2.9
\arrow <2mm> [0.25,0.75] from 1.15 2.9    to  1.65 3.52
\arrow <2mm> [0.25,0.75] from 1.05 2.9    to  .7 3.83 
\stoprotation

\startrotation by -0.5 -0.866 about 0 0
\multiput{$\circ$} at -.866 1.8 .866 1.8 /
\arrow <2mm> [0.25,0.75] from -0.1 1.1  to  -.78 1.75
\arrow <2mm> [0.25,0.75] from  0.1 1.1  to   .78 1.75
\multiput{$\circ$} at -2.13 2.1  -1.1 2.8 
  2.13 2.1  1.1 2.8 /
\arrow <2mm> [0.25,0.75] from -2. 2.1     to  -.95 1.8 
\arrow <2mm> [0.25,0.75] from -1.07 2.7   to  -.866 1.9
\arrow <2mm> [0.25,0.75] from 2. 2.1      to  .95 1.8 
\arrow <2mm> [0.25,0.75] from 1.07 2.7    to  .866 1.9

\multiput{$\circ$} at -3.27 2.3  -2.65 3 -1.7 3.62  -.7 3.93  /
\multiput{$\circ$} at  3.27 2.3  2.65 3  1.7 3.62  .7 3.93 /

\arrow <2mm> [0.25,0.75] from -2.3 2.14   to  -3.15 2.3 
\arrow <2mm> [0.25,0.75] from -2.15 2.2    to  -2.6 2.9
\arrow <2mm> [0.25,0.75] from -1.15 2.9    to  -1.65 3.52
\arrow <2mm> [0.25,0.75] from -1.05 2.9    to  -.7 3.83 

\arrow <2mm> [0.25,0.75] from 2.3 2.14   to  3.15 2.3 
\arrow <2mm> [0.25,0.75] from 2.15 2.2    to  2.6 2.9
\arrow <2mm> [0.25,0.75] from 1.15 2.9    to  1.65 3.52
\arrow <2mm> [0.25,0.75] from 1.05 2.9    to  .7 3.83 
\stoprotation

\setdots <1mm>
\circulararc 360 degrees from 1 0 center at 0 0
\circulararc 360 degrees from 2 0 center at 0 0
\circulararc 360 degrees from 3 0 center at 0 0
\circulararc 360 degrees from 4 0 center at 0 0
\circulararc 360 degrees from 5 0 center at 0 0
\endpicture} at   10 16.5   /

\multiput{\beginpicture

\put{\footnotesize 1} at 0 0
\multiput{\footnotesize  1} at -.866 -.5
  .866 -.5 /
\multiput{\footnotesize 0} at  0 1 /

\arrow <2mm> [0.25,0.75] from 0 0.9 to  0 0.2
\arrow <2mm> [0.25,0.75] from -.78 -.45
 to   -0.1 -0.05 
\arrow <2mm> [0.25,0.75] from .78 -.45
 to   0.1 -0.05 

\setshadegrid span <.5mm>
\vshade -.5 -.4 1.4  .5 -.4 1.4 /

\multiput{\footnotesize  0} at -.866 1.8 .866 1.8 /
\arrow <2mm> [0.25,0.75] from -0.1 1.1  to  -.78 1.75
\arrow <2mm> [0.25,0.75] from  0.1 1.1  to   .78 1.75
\multiput{$\circ$} at -2.13 2.1  -1.1 2.8 
  2.13 2.1  1.1 2.8 /
\arrow <2mm> [0.25,0.75] from -2. 2.1     to  -.95 1.8 
\arrow <2mm> [0.25,0.75] from -1.07 2.7   to  -.866 1.9
\arrow <2mm> [0.25,0.75] from 2. 2.1      to  .95 1.8 
\arrow <2mm> [0.25,0.75] from 1.07 2.7    to  .866 1.9

\multiput{$\circ$} at -3.27 2.3  -2.65 3 -1.7 3.62  -.7 3.93  /
\multiput{$\circ$} at  3.27 2.3  2.65 3  1.7 3.62  .7 3.93 /

\arrow <2mm> [0.25,0.75] from -2.3 2.14   to  -3.15 2.3 
\arrow <2mm> [0.25,0.75] from -2.15 2.2    to  -2.6 2.9
\arrow <2mm> [0.25,0.75] from -1.15 2.9    to  -1.65 3.52
\arrow <2mm> [0.25,0.75] from -1.05 2.9    to  -.7 3.83 

\arrow <2mm> [0.25,0.75] from 2.3 2.14   to  3.15 2.3 
\arrow <2mm> [0.25,0.75] from 2.15 2.2    to  2.6 2.9
\arrow <2mm> [0.25,0.75] from 1.15 2.9    to  1.65 3.52
\arrow <2mm> [0.25,0.75] from 1.05 2.9    to  .7 3.83 

\startrotation by -0.5 0.866 about 0 0
\multiput{\footnotesize  1} at -.866 1.8 .866 1.8 /
\arrow <2mm> [0.25,0.75] from -0.1 1.1  to  -.78 1.75
\arrow <2mm> [0.25,0.75] from  0.1 1.1  to   .78 1.75
\multiput{$\circ$} at -2.13 2.1  -1.1 2.8 
  2.13 2.1  1.1 2.8 /
\arrow <2mm> [0.25,0.75] from -2. 2.1     to  -.95 1.8 
\arrow <2mm> [0.25,0.75] from -1.07 2.7   to  -.866 1.9
\arrow <2mm> [0.25,0.75] from 2. 2.1      to  .95 1.8 
\arrow <2mm> [0.25,0.75] from 1.07 2.7    to  .866 1.9

\multiput{$\circ$} at -3.27 2.3  -2.65 3 -1.7 3.62  -.7 3.93  /
\multiput{$\circ$} at  3.27 2.3  2.65 3  1.7 3.62  .7 3.93 /

\arrow <2mm> [0.25,0.75] from -2.3 2.14   to  -3.15 2.3 
\arrow <2mm> [0.25,0.75] from -2.15 2.2    to  -2.6 2.9
\arrow <2mm> [0.25,0.75] from -1.15 2.9    to  -1.65 3.52
\arrow <2mm> [0.25,0.75] from -1.05 2.9    to  -.7 3.83 

\arrow <2mm> [0.25,0.75] from 2.3 2.14   to  3.15 2.3 
\arrow <2mm> [0.25,0.75] from 2.15 2.2    to  2.6 2.9
\arrow <2mm> [0.25,0.75] from 1.15 2.9    to  1.65 3.52
\arrow <2mm> [0.25,0.75] from 1.05 2.9    to  .7 3.83 
\stoprotation

\startrotation by -0.5 -0.866 about 0 0
\multiput{\footnotesize  1} at -.866 1.8 .866 1.8 /
\arrow <2mm> [0.25,0.75] from -0.1 1.1  to  -.78 1.75
\arrow <2mm> [0.25,0.75] from  0.1 1.1  to   .78 1.75
\multiput{$\circ$} at -2.13 2.1  -1.1 2.8 
  2.13 2.1  1.1 2.8 /
\arrow <2mm> [0.25,0.75] from -2. 2.1     to  -.95 1.8 
\arrow <2mm> [0.25,0.75] from -1.07 2.7   to  -.866 1.9
\arrow <2mm> [0.25,0.75] from 2. 2.1      to  .95 1.8 
\arrow <2mm> [0.25,0.75] from 1.07 2.7    to  .866 1.9

\multiput{$\circ$} at -3.27 2.3  -2.65 3 -1.7 3.62  -.7 3.93  /
\multiput{$\circ$} at  3.27 2.3  2.65 3  1.7 3.62  .7 3.93 /

\arrow <2mm> [0.25,0.75] from -2.3 2.14   to  -3.15 2.3 
\arrow <2mm> [0.25,0.75] from -2.15 2.2    to  -2.6 2.9
\arrow <2mm> [0.25,0.75] from -1.15 2.9    to  -1.65 3.52
\arrow <2mm> [0.25,0.75] from -1.05 2.9    to  -.7 3.83 

\arrow <2mm> [0.25,0.75] from 2.3 2.14   to  3.15 2.3 
\arrow <2mm> [0.25,0.75] from 2.15 2.2    to  2.6 2.9
\arrow <2mm> [0.25,0.75] from 1.15 2.9    to  1.65 3.52
\arrow <2mm> [0.25,0.75] from 1.05 2.9    to  .7 3.83 
\stoprotation

\setdots <1mm>
\circulararc 360 degrees from 1 0 center at 0 0
\circulararc 360 degrees from 2 0 center at 0 0
\circulararc 360 degrees from 3 0 center at 0 0
\circulararc 360 degrees from 4 0 center at 0 0
\circulararc 360 degrees from 5 0 center at 0 0
\endpicture} at 0 11  /

\multiput{\beginpicture

\put{\footnotesize  1} at 0 0
\multiput{\footnotesize  2} at -.866 -.5
  .866 -.5  /
\multiput{\footnotesize  1} at 0 1 /

\arrow <2mm> [0.25,0.75] from 0 0.9 to  0 0.2
\arrow <2mm> [0.25,0.75] from -.78 -.45
 to   -0.1 -0.05 
\arrow <2mm> [0.25,0.75] from .78 -.45
 to   0.1 -0.05 

\setshadegrid span <.5mm>
\vshade -.5 -.4 1.4  .5 -.4 1.4 /

\multiput{\footnotesize  0} at -.866 1.8 .866 1.8 /
\arrow <2mm> [0.25,0.75] from -0.1 1.1  to  -.78 1.75
\arrow <2mm> [0.25,0.75] from  0.1 1.1  to   .78 1.75
\multiput{\footnotesize  0} at -2.13 2.1  -1.1 2.8 
  2.13 2.1  1.1 2.8 /
\arrow <2mm> [0.25,0.75] from -2. 2.1     to  -.95 1.8 
\arrow <2mm> [0.25,0.75] from -1.07 2.7   to  -.866 1.9
\arrow <2mm> [0.25,0.75] from 2. 2.1      to  .95 1.8 
\arrow <2mm> [0.25,0.75] from 1.07 2.7    to  .866 1.9

\multiput{$\circ$} at -3.27 2.3  -2.65 3 -1.7 3.62  -.7 3.93  /
\multiput{$\circ$} at  3.27 2.3  2.65 3  1.7 3.62  .7 3.93 /

\arrow <2mm> [0.25,0.75] from -2.3 2.14   to  -3.15 2.3 
\arrow <2mm> [0.25,0.75] from -2.15 2.2    to  -2.6 2.9
\arrow <2mm> [0.25,0.75] from -1.15 2.9    to  -1.65 3.52
\arrow <2mm> [0.25,0.75] from -1.05 2.9    to  -.7 3.83 

\arrow <2mm> [0.25,0.75] from 2.3 2.14   to  3.15 2.3 
\arrow <2mm> [0.25,0.75] from 2.15 2.2    to  2.6 2.9
\arrow <2mm> [0.25,0.75] from 1.15 2.9    to  1.65 3.52
\arrow <2mm> [0.25,0.75] from 1.05 2.9    to  .7 3.83 

\startrotation by -0.5 0.866 about 0 0
\multiput{\footnotesize  1} at -.866 1.8 .866 1.8 /
\arrow <2mm> [0.25,0.75] from -0.1 1.1  to  -.78 1.75
\arrow <2mm> [0.25,0.75] from  0.1 1.1  to   .78 1.75
\multiput{\footnotesize  1} at -2.13 2.1  -1.1 2.8 
  2.13 2.1  1.1 2.8 /
\arrow <2mm> [0.25,0.75] from -2. 2.1     to  -.95 1.8 
\arrow <2mm> [0.25,0.75] from -1.07 2.7   to  -.866 1.9
\arrow <2mm> [0.25,0.75] from 2. 2.1      to  .95 1.8 
\arrow <2mm> [0.25,0.75] from 1.07 2.7    to  .866 1.9

\multiput{$\circ$} at -3.27 2.3  -2.65 3 -1.7 3.62  -.7 3.93  /
\multiput{$\circ$} at  3.27 2.3  2.65 3  1.7 3.62  .7 3.93 /

\arrow <2mm> [0.25,0.75] from -2.3 2.14   to  -3.15 2.3 
\arrow <2mm> [0.25,0.75] from -2.15 2.2    to  -2.6 2.9
\arrow <2mm> [0.25,0.75] from -1.15 2.9    to  -1.65 3.52
\arrow <2mm> [0.25,0.75] from -1.05 2.9    to  -.7 3.83 

\arrow <2mm> [0.25,0.75] from 2.3 2.14   to  3.15 2.3 
\arrow <2mm> [0.25,0.75] from 2.15 2.2    to  2.6 2.9
\arrow <2mm> [0.25,0.75] from 1.15 2.9    to  1.65 3.52
\arrow <2mm> [0.25,0.75] from 1.05 2.9    to  .7 3.83 
\stoprotation

\startrotation by -0.5 -0.866 about 0 0
\multiput{\footnotesize  1} at -.866 1.8 .866 1.8 /
\arrow <2mm> [0.25,0.75] from -0.1 1.1  to  -.78 1.75
\arrow <2mm> [0.25,0.75] from  0.1 1.1  to   .78 1.75
\multiput{\footnotesize  1} at -2.13 2.1  -1.1 2.8 
  2.13 2.1  1.1 2.8 /
\arrow <2mm> [0.25,0.75] from -2. 2.1     to  -.95 1.8 
\arrow <2mm> [0.25,0.75] from -1.07 2.7   to  -.866 1.9
\arrow <2mm> [0.25,0.75] from 2. 2.1      to  .95 1.8 
\arrow <2mm> [0.25,0.75] from 1.07 2.7    to  .866 1.9

\multiput{$\circ$} at -3.27 2.3  -2.65 3 -1.7 3.62  -.7 3.93  /
\multiput{$\circ$} at  3.27 2.3  2.65 3  1.7 3.62  .7 3.93 /

\arrow <2mm> [0.25,0.75] from -2.3 2.14   to  -3.15 2.3 
\arrow <2mm> [0.25,0.75] from -2.15 2.2    to  -2.6 2.9
\arrow <2mm> [0.25,0.75] from -1.15 2.9    to  -1.65 3.52
\arrow <2mm> [0.25,0.75] from -1.05 2.9    to  -.7 3.83 

\arrow <2mm> [0.25,0.75] from 2.3 2.14   to  3.15 2.3 
\arrow <2mm> [0.25,0.75] from 2.15 2.2    to  2.6 2.9
\arrow <2mm> [0.25,0.75] from 1.15 2.9    to  1.65 3.52
\arrow <2mm> [0.25,0.75] from 1.05 2.9    to  .7 3.83 
\stoprotation

\setdots <1mm>
\circulararc 360 degrees from 1 0 center at 0 0
\circulararc 360 degrees from 2 0 center at 0 0
\circulararc 360 degrees from 3 0 center at 0 0
\circulararc 360 degrees from 4 0 center at 0 0
\circulararc 360 degrees from 5 0 center at 0 0
\endpicture} at    10 5.5    /

\multiput{\beginpicture

\put{\footnotesize 4} at 0 0
\multiput{\footnotesize 2} at -.866 -.5
  .866 -.5  /
\multiput{\footnotesize 1} at  0 1 /

\arrow <2mm> [0.25,0.75] from 0 0.9 to  0 0.2
\arrow <2mm> [0.25,0.75] from -.78 -.45
 to   -0.1 -0.05 
\arrow <2mm> [0.25,0.75] from .78 -.45
 to   0.1 -0.05 

\setshadegrid span <.5mm>
\vshade -.5 -.4 1.4  .5 -.4 1.4 /

\multiput{\footnotesize 1} at -.866 1.8 .866 1.8 /
\arrow <2mm> [0.25,0.75] from -0.1 1.1  to  -.78 1.75
\arrow <2mm> [0.25,0.75] from  0.1 1.1  to   .78 1.75
\multiput{\footnotesize 0} at -2.13 2.1  -1.1 2.8 
  2.13 2.1  1.1 2.8 /
\arrow <2mm> [0.25,0.75] from -2. 2.1     to  -.95 1.8 
\arrow <2mm> [0.25,0.75] from -1.07 2.7   to  -.866 1.9
\arrow <2mm> [0.25,0.75] from 2. 2.1      to  .95 1.8 
\arrow <2mm> [0.25,0.75] from 1.07 2.7    to  .866 1.9

\multiput{\footnotesize 0} at -3.27 2.3  -2.65 3 -1.7 3.62  -.7 3.93  /
\multiput{\footnotesize 0} at  3.27 2.3  2.65 3  1.7 3.62  .7 3.93 /                          

\arrow <2mm> [0.25,0.75] from -2.3 2.14   to  -3.15 2.3 
\arrow <2mm> [0.25,0.75] from -2.15 2.2    to  -2.6 2.9
\arrow <2mm> [0.25,0.75] from -1.15 2.9    to  -1.65 3.52
\arrow <2mm> [0.25,0.75] from -1.05 2.9    to  -.7 3.83 

\arrow <2mm> [0.25,0.75] from 2.3 2.14   to  3.15 2.3 
\arrow <2mm> [0.25,0.75] from 2.15 2.2    to  2.6 2.9
\arrow <2mm> [0.25,0.75] from 1.15 2.9    to  1.65 3.52
\arrow <2mm> [0.25,0.75] from 1.05 2.9    to  .7 3.83 

\startrotation by -0.5 0.866 about 0 0
\multiput{\footnotesize 3} at -.866 1.8 .866 1.8 /
\arrow <2mm> [0.25,0.75] from -0.1 1.1  to  -.78 1.75
\arrow <2mm> [0.25,0.75] from  0.1 1.1  to   .78 1.75
\multiput{\footnotesize 1} at -2.13 2.1  -1.1 2.8 
  2.13 2.1  1.1 2.8 /
\arrow <2mm> [0.25,0.75] from -2. 2.1     to  -.95 1.8 
\arrow <2mm> [0.25,0.75] from -1.07 2.7   to  -.866 1.9
\arrow <2mm> [0.25,0.75] from 2. 2.1      to  .95 1.8 
\arrow <2mm> [0.25,0.75] from 1.07 2.7    to  .866 1.9

\multiput{\footnotesize 1} at -3.27 2.3  -2.65 3 -1.7 3.62  -.7 3.93  /
\multiput{\footnotesize 1} at  3.27 2.3  2.65 3  1.7 3.62  .7 3.93 /

\arrow <2mm> [0.25,0.75] from -2.3 2.14   to  -3.15 2.3 
\arrow <2mm> [0.25,0.75] from -2.15 2.2    to  -2.6 2.9
\arrow <2mm> [0.25,0.75] from -1.15 2.9    to  -1.65 3.52
\arrow <2mm> [0.25,0.75] from -1.05 2.9    to  -.7 3.83 

\arrow <2mm> [0.25,0.75] from 2.3 2.14   to  3.15 2.3 
\arrow <2mm> [0.25,0.75] from 2.15 2.2    to  2.6 2.9
\arrow <2mm> [0.25,0.75] from 1.15 2.9    to  1.65 3.52
\arrow <2mm> [0.25,0.75] from 1.05 2.9    to  .7 3.83 
\stoprotation

\startrotation by -0.5 -0.866 about 0 0
\multiput{\footnotesize 3} at -.866 1.8 .866 1.8 /
\arrow <2mm> [0.25,0.75] from -0.1 1.1  to  -.78 1.75
\arrow <2mm> [0.25,0.75] from  0.1 1.1  to   .78 1.75
\multiput{\footnotesize 1} at -2.13 2.1  -1.1 2.8 
  2.13 2.1  1.1 2.8 /
\arrow <2mm> [0.25,0.75] from -2. 2.1     to  -.95 1.8 
\arrow <2mm> [0.25,0.75] from -1.07 2.7   to  -.866 1.9
\arrow <2mm> [0.25,0.75] from 2. 2.1      to  .95 1.8 
\arrow <2mm> [0.25,0.75] from 1.07 2.7    to  .866 1.9

\multiput{\footnotesize 1} at -3.27 2.3  -2.65 3 -1.7 3.62  -.7 3.93  /
\multiput{\footnotesize 1} at  3.27 2.3  2.65 3  1.7 3.62  .7 3.93 /

\arrow <2mm> [0.25,0.75] from -2.3 2.14   to  -3.15 2.3 
\arrow <2mm> [0.25,0.75] from -2.15 2.2    to  -2.6 2.9
\arrow <2mm> [0.25,0.75] from -1.15 2.9    to  -1.65 3.52
\arrow <2mm> [0.25,0.75] from -1.05 2.9    to  -.7 3.83 

\arrow <2mm> [0.25,0.75] from 2.3 2.14   to  3.15 2.3 
\arrow <2mm> [0.25,0.75] from 2.15 2.2    to  2.6 2.9
\arrow <2mm> [0.25,0.75] from 1.15 2.9    to  1.65 3.52
\arrow <2mm> [0.25,0.75] from 1.05 2.9    to  .7 3.83 
\stoprotation

\setdots <1mm>
\circulararc 360 degrees from 1 0 center at 0 0
\circulararc 360 degrees from 2 0 center at 0 0
\circulararc 360 degrees from 3 0 center at 0 0
\circulararc 360 degrees from 4 0 center at 0 0
\circulararc 360 degrees from 5 0 center at 0 0
\endpicture} at 0 0    /

\multiput{\beginpicture

\put{\footnotesize 4} at 0 0
\multiput{\footnotesize 8} at -.866 -.5
  .866 -.5  /
\multiput{\footnotesize 5} at 0 1 /

\arrow <2mm> [0.25,0.75] from 0 0.9 to  0 0.2
\arrow <2mm> [0.25,0.75] from -.78 -.45
 to   -0.1 -0.05 
\arrow <2mm> [0.25,0.75] from .78 -.45
 to   0.1 -0.05 

\setshadegrid span <.5mm>
\vshade -.5 -.4 1.4  .5 -.4 1.4 /

\multiput{\footnotesize 1} at -.866 1.8 .866 1.8 /
\arrow <2mm> [0.25,0.75] from -0.1 1.1  to  -.78 1.75
\arrow <2mm> [0.25,0.75] from  0.1 1.1  to   .78 1.75
\multiput{\footnotesize 1} at -2.13 2.1  -1.1 2.8 
  2.13 2.1  1.1 2.8 /
\arrow <2mm> [0.25,0.75] from -2. 2.1     to  -.95 1.8 
\arrow <2mm> [0.25,0.75] from -1.07 2.7   to  -.866 1.9
\arrow <2mm> [0.25,0.75] from 2. 2.1      to  .95 1.8 
\arrow <2mm> [0.25,0.75] from 1.07 2.7    to  .866 1.9

\multiput{\footnotesize 0} at -3.27 2.3  -2.65 3 -1.7 3.62  -.7 3.93  /
\multiput{\footnotesize 0} at  3.27 2.3  2.65 3  1.7 3.62  .7 3.93 /

\arrow <2mm> [0.25,0.75] from -2.3 2.14   to  -3.15 2.3 
\arrow <2mm> [0.25,0.75] from -2.15 2.2    to  -2.6 2.9
\arrow <2mm> [0.25,0.75] from -1.15 2.9    to  -1.65 3.52
\arrow <2mm> [0.25,0.75] from -1.05 2.9    to  -.7 3.83 

\arrow <2mm> [0.25,0.75] from 2.3 2.14   to  3.15 2.3 
\arrow <2mm> [0.25,0.75] from 2.15 2.2    to  2.6 2.9
\arrow <2mm> [0.25,0.75] from 1.15 2.9    to  1.65 3.52
\arrow <2mm> [0.25,0.75] from 1.05 2.9    to  .7 3.83 

\startrotation by -0.5 0.866 about 0 0
\multiput{\footnotesize 3} at -.866 1.8 .866 1.8 /
\arrow <2mm> [0.25,0.75] from -0.1 1.1  to  -.78 1.75
\arrow <2mm> [0.25,0.75] from  0.1 1.1  to   .78 1.75
\multiput{\footnotesize 4} at -2.13 2.1  -1.1 2.8 
  2.13 2.1  1.1 2.8 /
\arrow <2mm> [0.25,0.75] from -2. 2.1     to  -.95 1.8 
\arrow <2mm> [0.25,0.75] from -1.07 2.7   to  -.866 1.9
\arrow <2mm> [0.25,0.75] from 2. 2.1      to  .95 1.8 
\arrow <2mm> [0.25,0.75] from 1.07 2.7    to  .866 1.9

\multiput{\footnotesize 1} at -3.27 2.3  -2.65 3 -1.7 3.62  -.7 3.93  /
\multiput{\footnotesize 1} at  3.27 2.3  2.65 3  1.7 3.62  .7 3.93 /

\arrow <2mm> [0.25,0.75] from -2.3 2.14   to  -3.15 2.3 
\arrow <2mm> [0.25,0.75] from -2.15 2.2    to  -2.6 2.9
\arrow <2mm> [0.25,0.75] from -1.15 2.9    to  -1.65 3.52
\arrow <2mm> [0.25,0.75] from -1.05 2.9    to  -.7 3.83 

\arrow <2mm> [0.25,0.75] from 2.3 2.14   to  3.15 2.3 
\arrow <2mm> [0.25,0.75] from 2.15 2.2    to  2.6 2.9
\arrow <2mm> [0.25,0.75] from 1.15 2.9    to  1.65 3.52
\arrow <2mm> [0.25,0.75] from 1.05 2.9    to  .7 3.83 
\stoprotation

\startrotation by -0.5 -0.866 about 0 0
\multiput{\footnotesize 3} at -.866 1.8 .866 1.8 /
\arrow <2mm> [0.25,0.75] from -0.1 1.1  to  -.78 1.75
\arrow <2mm> [0.25,0.75] from  0.1 1.1  to   .78 1.75
\multiput{\footnotesize 4} at -2.13 2.1  -1.1 2.8 
  2.13 2.1  1.1 2.8 /
\arrow <2mm> [0.25,0.75] from -2. 2.1     to  -.95 1.8 
\arrow <2mm> [0.25,0.75] from -1.07 2.7   to  -.866 1.9
\arrow <2mm> [0.25,0.75] from 2. 2.1      to  .95 1.8 
\arrow <2mm> [0.25,0.75] from 1.07 2.7    to  .866 1.9

\multiput{\footnotesize 1} at -3.27 2.3  -2.65 3 -1.7 3.62  -.7 3.93  /
\multiput{\footnotesize 1} at  3.27 2.3  2.65 3  1.7 3.62  .7 3.93 /

\arrow <2mm> [0.25,0.75] from -2.3 2.14   to  -3.15 2.3 
\arrow <2mm> [0.25,0.75] from -2.15 2.2    to  -2.6 2.9
\arrow <2mm> [0.25,0.75] from -1.15 2.9    to  -1.65 3.52
\arrow <2mm> [0.25,0.75] from -1.05 2.9    to  -.7 3.83 

\arrow <2mm> [0.25,0.75] from 2.3 2.14   to  3.15 2.3 
\arrow <2mm> [0.25,0.75] from 2.15 2.2    to  2.6 2.9
\arrow <2mm> [0.25,0.75] from 1.15 2.9    to  1.65 3.52
\arrow <2mm> [0.25,0.75] from 1.05 2.9    to  .7 3.83 
\stoprotation

\setdots <1mm>
\circulararc 360 degrees from 1 0 center at 0 0
\circulararc 360 degrees from 2 0 center at 0 0
\circulararc 360 degrees from 3 0 center at 0 0
\circulararc 360 degrees from 4 0 center at 0 0
\circulararc 360 degrees from 5 0 center at 0 0

\multiput{\footnotesize 1} at -5 -.5 -5 0 -4.95 0.5  -4.9 1.2 -4.7 1.8 
 -4.5 2.3 
 5 -.5  5 0  4.95 0.5   4.9 1.2  4.7 1.8 
 4.5 2.3
  /
\multiput{\footnotesize 0} at 
 -4.3 2.7 -4  3.1  -3.6 3.5 -3.2 3.9 
 -2.4 4.4  -2. 4.6  -1.5 4.8  -0.7 4.95 
   4.3 2.7  4 3.1 
3.6 3.5  3.2 3.9 
 2.4 4.4  2. 4.6  1.5 4.8  0.7 4.95
 /

\endpicture} at        10 -5.5 /

\put{} at -7  25
\put{$r_0(x,y)$} at -7 22
\put{$[1,1]$} at -7 21
\put{$r_1(x,y)$} at -7 16.5
\put{$[1,2]$} at -7 15.5
\put{$r_2(x,y)$} at -7 11
\put{$[2,5]$} at -7 10
\put{$r_3(x,y)$} at -7 5.5
\put{$[5,13]$} at -7 4.5
\put{$r_4(x,y)$} at -7 0
\put{$[13,34]$} at -7 -1
\put{$r_5(x,y)$} at -7 -5.5
\put{$[34,89]$} at -7 -6.5
\endpicture}
$$

\end{document}